\documentclass[review]{elsarticle}

\usepackage{amssymb}
\usepackage{amsmath}
\usepackage{amsthm}
\usepackage{amsfonts}
\usepackage{bm}
\usepackage{mathrsfs}
\usepackage{curves}
\usepackage{color}
\usepackage{subfig}
\usepackage{booktabs}
\usepackage{bookmark}
\usepackage{multirow}

\usepackage{lineno,hyperref}
\modulolinenumbers[5]

\newtheorem{theorem}{Theorem}[section]
\newtheorem{lemma}{Lemma}[section]

\newdefinition{definition}{Definition}[section]
\newdefinition{example}{Example}[section]
\numberwithin{equation}{section}

\begin{document}
\begin{frontmatter}

\title{Maximum-norm error analysis of compact difference schemes for the backward fractional Feynman-Kac equation }
% \renewcommand{\thefootnote}{\fnsymbol{footnote}}
% \footnotetext[2]{Address}
\author[nwpu,hut]{Jiahui Hu}
\ead{hujh@mail.nwpu.edu.cn}
\author[nwpu]{Jungang Wang}
\author[nwpu]{Zhanbin Yuan}
\author[nwpu]{Zongze Yang}
%\author[nwpu]{Xiaogang Zhu}
\author[nwpu]{Yufeng Nie\corref{cor1}}
\ead{yfnie@nwpu.edu.cn}

\cortext[cor1]{Corresponding author}
\address[nwpu]{Research Center for Computational Science, Northwestern Polytechnical University, Xi'an 710129, China}
\address[hut]{College of Science, Henan University of Technology, Zhengzhou 450001, China}
\begin{abstract}
  %In this paper, we consider two dimensional Riesz space fractional
%  diffusion equations with nonlinear source term on complex convex domain,
%  which has been solved by finite element methods on regular domain with structured meshes,
%  mostly uniform meshes, in existing literature.
%  To break this limitation, we perform finite element method on unstructured triangular meshes,
%  which is convenient to process problems on complex domain.
%  Stability and convergence of backward Euler Galerkin method % fully discrete scheme
%  have been discussed.
%  Several numerical examples are presented to verify theoretical analysis.

  The fractional Feynman-Kac equations describe the
  distribution of functionals of non-Brownian motion, or
  anomalous diffusion, including two types called the forward and backward fractional Feynman-Kac equations,
  where the fractional substantial derivative is involved.
  This paper focuses on the more widely used backward version. Based on the discretized schemes for fractional substantial
  derivatives proposed recently, we construct compact finite difference schemes for the backward fractional
  Feynman-Kac equation, which has $q$-th $(q=1,2,3,4)$ order accuracy in temporal direction and fourth order accuracy in spatial
  direction, respectively. In the case $q=1$, the numerical stability and convergence of the difference scheme
  %with first-order accuracy in temporal direction
  in the discrete $L^\infty$ norm are proved strictly, where a new inner product is defined for the theoretical analysis.
  Finally, numerical examples are provided to verify the effectiveness and accuracy
  of the algorithms.

  %which is about
  %the joint probability density function (PDF) of finding the particle on $A$ at time $t$ with the initial
  %position of the particle at $x$.
\end{abstract}
\begin{keyword}
Backward fractional Feynman-Kac equation \sep
Fractional substantial derivative \sep
  Compact scheme \sep
  Maximum-norm error analysis\sep
  Stability \sep
  Convergence
  \MSC[2010] 35R11\sep 65M06\sep 65M12
\end{keyword}
\end{frontmatter}
%\begin{keywords}finite element method, fractional derivative, nonlinear,
  %polyhedron domain\end{keywords}
%\begin{MSC}
%35R11\sep
%65M06\sep
%65M12
%\end{MSC}

%\markboth{Fractional Diffusion Equations}{Finite Element Methods}

\section[Introduction]{Introduction}\label{sec:introduciton}
Diffusive motions exist widely in the nature, among the fields from condensed matter physics \cite{comtet2005functionals, FoltinOerdingRaczEtAl1994,HummerSzabo2001},
to hydrodynamics \cite{BauleFriedrich2006}, meteorology \cite{MajumdarBray2002}, and finance \cite{Yor2001,Yor2012}.
%Hence the functionals of the processes arise accordingly.
Thus the Brownian functionals play important role in science community.
Assume $x(t)$ is a path of a Brownian particle in the time interval $(0,t)$, and $U(x)$ is some prescribed function.
Then a Brownian functional can be defined as $A=\int_0^tU[x(\tau)]d\tau$. $A$ is a random variable for $x(t)$ is a random path.
On account of the diversity of the function $U(x)$ being chosen, the Brownian functional $A$ models different phenomena.
%For example, fluctuating interfaces [sub2,2] as $U(x)=x^2$, Obukhov's model of advection of particles in turbulent flow [sub2,3,4] when
%$U(x)=x$ and the velocity field is modeled with Brownian motion, and
%finance of stock prices for $U(x)=\exp(-\beta x)$, and so on.
In 1949, by using Feynman's path integral method Kac derived the (imaginary time) Schr\"{o}dinger equation for
the distribution function of $A$ \cite{kac1949}. %Denote by $P(x,A,t)$ the
%joint probability density function (PDF) of finding the particle at $x$ and the functional at $A$ at time $t$.
%%In 1949, Kac was enlightened by Feyman's path integrals to derive the distribution of the functionals which satisfies a Schr\"{o}dinger-like equation. Denote by $G(x,A,t)$ the joint probability density function (PDF) of finding the particle at $x$ and the functional at $A$ at time $t$.
%The celebrated Feynman-Kac equation is described as
%\begin{equation}\label{eq:fk}
%\frac{\partial}{\partial t}P(x,\rho,t)=K\frac{\partial^2}{\partial x^2}P(x,\rho,t)-\rho U(x)P(x,\rho,t)
%\end{equation}
%where $P(x,\rho,t):=\int_0^\infty P(x,A,t)e^{-\rho A}dA$, $\mathfrak{R}{(\rho)}>0$, $U(x)\geq0$,
%and the diffusion coefficient $K$ is a positive constant.
%Equation~\eqref{eq:fk} represents the functional of normal Brownian motion.
However, in recent years, by realizing that
%the diffusion of particles is anomalously slow in many systems,
numerous anomalous diffusion phenomena exist widely in many systems,
the scientists pay more and more attention to the anomalous diffusion processes,
or the non-Brownian motion, which can be modeled more exactly by fractional differential equations.
%which are reflected by the nonlinear relation $\langle x^2\rangle\sim t^\alpha$.
%for $0<\alpha<1$.
%To model the anomalous diffusion phenomena exactly, fractional differential equations have been used in the relevant issues.

Let $x(t)$ be a trajectory of non-Brownian particle.
The functional of anomalous diffusion has the same form as the Brownian functional
\begin{equation}\label{eq:func}
A=\int_0^t U[x(\tau)]d\tau.
\end{equation}
%where $x(t)$ is a trajectory of non-Brownian particle.
For the different prescribed function $U(x)$, various
physics processes can be characterized. For instance, when taking $U(x)=1$ in a given domain and to be zero otherwise,
$A$ models the time spent by a particle in the domain. The corresponding functional can be used in kinetic studies of chemical
reactions that take place exclusively in the domain \cite{Agmon1984,CarmiTurgemanBarkai2010}. When the motion of the particles is non-Brownian in dispersive systems with inhomogeneous disorder, we take $U(x)=x$ or $x^2$ \cite{CarmiTurgemanBarkai2010}.
By employing a versatile framework for describing the motion of particles in disordered systems, i.e.,
the continuous time random walk (CTRW),
%For developing a general theory of non-Brownian functionals,
Carmi, Turgeman, and Barkai
%employed a versatile framework for describing the motion of particles in disordered systems, i.e., the continuous time random walk (CTRW).
derived the forward and backward fractional Feynman-Kac equations \cite{CarmiTurgemanBarkai2010,CarmiBarkai2011,TurgemanCarmiBarkai2009},
where the fractional substantial derivative is involved.
Both forward and backward fractional Feynman-Kac equations describe the
distributions of functionals of the widely observed subdiffusive processes.
%In most cases, the backward version is more convenient to use since it represent the distribution of the
%functional $A$ directly while the forward one characterizes the joint
While in most cases, scholars are only interested in the distribution of the functional $A$ and
regardless of the final position of the particle, $x$,
it turns out to be more convenient to use the backward version.
%it turns out to be quite inconvenient to get the PDF of $A$ by integrating $P$ over all $x$.
%therefore the more widely used version is for the PDF of $A$ at time $t$, namely,
The backward
fractional Feynman-Kac equation which is shown as following, will be discussed detailedly in our work.
Denote by $P(x,A,t)$ the probability density function (PDF) of $A$ at time $t$, given that the process has started at $x$.
The backward fractional Feynman-Kac equation is given as \cite{CarmiTurgemanBarkai2010,CarmiBarkai2011,TurgemanCarmiBarkai2009}
%the forward fractional Feynman-Kac equation can be written as
%\begin{equation}\label{eq:ffk}
%\frac{\partial}{\partial t}P(x,\rho,t)=K_\alpha \frac{\partial^2}{\partial x^2}{}^{s}D_t^{1-\alpha}P(x,\rho,t)
%-\rho U(x)P(x,\rho,t),
%\end{equation}
\begin{equation}\label{eq:bfkk}
\frac{\partial}{\partial t}P(x,\rho,t)=K_\alpha {}^{s}D_t^{1-\alpha}\frac{\partial^2}{\partial x^2}P(x,\rho,t)
-\rho U(x)P(x,\rho,t),
\end{equation}
where $P(x,\rho,t):=\int_0^\infty P(x,A,t)e^{-\rho A}dA$, $\mathfrak{R}{(\rho)}>0$, $U(x)\geq0$;
the functional $A$ is defined as~\eqref{eq:func} and $\alpha \in (0,1)$;
the diffusion coefficient $K_\alpha$ is a positive constant, and
the symbol ${}^{s}D_t^{\nu}$ represents the Friedrich's fractional substantial derivative of order $\nu$
%and is equal to $\left[l+\rho U(x)\right]^{\nu}$ in Laplace $t\rightarrow l$ space
\cite{FriedrichJenkoBauleEtAl2006}.

For the past few years, the numerical methods for solving fractional partial differential equations (PDEs) have been well developed, including finite difference methods \cite{ChenDengWu2013,ChenLiuZhuangEtAl2009,MeerschaertTadjeran2006,SunWu2006,Gao2011},
finite element methods \cite{Deng2008,Ervin2006,JiangMa2011},
%finite volume method [],
and spectral methods \cite{LiZengLiu2012,LiXu2009}, etc.
However, for the PDEs with fractional substantial derivative, though there have been some works for getting the numerical
solutions \cite{DengChenBarkai2015}, high order finite difference schemes with the maximum norm error estimates
are still scarce. As is well-known, high order schemes lead to
more accurate results if the solution of the equation is regular enough.
The fractional substantial derivative is a non-local time-space coupled operator, which makes numerically solving the corresponding equations more difficult than other fractional PDEs, especially when using high order schemes. Besides,
compared with the error estimates in discrete $L^2$ norm, the discrete $L^\infty$ norm error estimates provide more immediate insight
on the error occurring during time evolution. Thus, in practice,
%when we measure the computation errors,
error estimates
in the grid independent maximum norm are preferred in numerical analysis.
The purpose of this paper is to develop high order compact difference schemes for the backward fractional
Feynman-Kac equation
and provide a rigorous
error analysis in the discrete $L^\infty$ norm by the strategy of introducing some kind of new inner product and corresponding norms.

%As is well-known, high order schemes lead to
%more accurate results if the solution of the equation is regular enough, since
%the high order schemes can keep the same computational cost as the first order ones but greatly improve the
%accuracy.
%Furthermore, the fractional substantial derivative is a non-local time-space coupled operator, which makes numerically solving the corresponding equations more difficult than the other fractional PDEs, especially when using high order schemes.
%In this paper, we construct compact difference schemes for solving the backward fractional Feynman-Kac equations by using the idea proposed in \cite{ChenDeng2013} to discretize the fractional substantial derivative in temporal direction.
%%and fourth order compact difference scheme for the spatial direction, respectively.
%For the proposed first order time discretization scheme, we theoretically prove that it is unconditionally stable, and the convergence order is $\mathcal{O}(\tau+h^4)$ in discrete $L^\infty$ norm,
%which makes the results stronger than $L^2$ norm does. Examples are provided to verify the convergence orders of the compact difference schemes from first to fourth order in temporal direction and fourth order in spatial direction.
%%for the backward Feynman-Kac equation.

The definitions of fractional substantial calculus are given as follows \cite{ChenDeng2013}.
\begin{definition}
Let $\nu >0$, $\rho$ be a constant, and $P(t)$ be piecewise continuous on $(0,\infty)$ and integrable
on any finite subinterval of $[0,\infty)$. Then the fractional substantial integral of $P(t)$ of order $\nu$ is defined as
\begin{equation}
{}^{s}I_t^\nu P(t)=\frac{1}{\Gamma(\nu)}\int_0^t (t-\tau)^{\nu-1}e^{-\rho U(x)(t-\tau)}P(\tau)d\tau,\qquad t>0,
\end{equation}
where $U(x)$ is a prescribed function in~\eqref{eq:func}.
\end{definition}

\begin{definition}
Let $\mu >0$, $\rho$ be a constant, and $P(t)$ be $(m-1)$-times continuously differentiable on $(0,\infty)$ and its $m$-times derivative be integrable on any finite subinterval of $[0,\infty)$, where $m$ is the smallest integer that exceeds $\mu$. Then the fractional substantial derivative of $P(t)$ of order $\mu$ is defined as
\begin{equation}
{}^sD_t^\mu P(t)={}^sD_t^m\lbrack {}^sI_t^{m-\mu}P(t)\rbrack,
\end{equation}
where
\begin{equation}
{}^sD_t^m=\left(\frac{\partial}{\partial t}+\rho U(x)\right)^m.
\end{equation}
\end{definition}

According to the definition of fractional substantial derivative, equation~\eqref{eq:bfkk} can be expressed as
the following form
\begin{equation}\label{eq:ebfkk}
{}^sD_t^1 P(x,\rho,t)={}^sD_t^{1-\alpha}\left [ K_\alpha\frac{\partial^2}{\partial x^2}P(x,\rho,t)\right ].
\end{equation}
Denote by ${}_c^sD_t^\alpha$ the Caputo fractional substantial derivative \cite{ChenDeng2013}, i.e.,
\begin{equation}
\begin{aligned}
{}_c^sD_t^\alpha P(x,t)&:={}^sD_t^\alpha\left[P(x,t)-e^{-\rho U(x)t}P(x,0)\right]\\
&={}^sD_t^\alpha P(x,t)-\frac{t^{-\alpha}e^{-\rho U(x)t}}{\Gamma(1-\alpha)}P(x,0),\nonumber
\end{aligned}
\end{equation}
then the equivalent form of~\eqref{eq:ebfkk} can be written as \cite{DengChenBarkai2015}
\begin{equation}
\begin{aligned}
{}_c^sD_t^\alpha P(x,t)=K_\alpha \frac{\partial^2}{\partial x^2}P(x,t),
\end{aligned}
\end{equation}
where $P(x,\rho,t)$ is replaced by $P(x,t)$ since $\rho$ is given as a fixed constant.
In the following, we still use $P(x,t)$ for convenience.

The remainder of this paper is organized as follows. In Section 2, we construct the compact finite difference
schemes for the backward fractional Feynman-Kac equation. In Section 3, by introducing some kind of inner product and
norms, we prove the stability and convergence of the first order time discretization
scheme in the discrete $L^\infty$ norm rigorously. Numerical examples are provided
to verify the effectiveness and accuracy of the proposed compact schemes from first to fourth order time discretization
in Section 4. Finally we draw some conclusions in the last
section.

\section[compact scheme]{Compact finite difference schemes for the backward fractional Feynman-Kac equation}\label{sec:scheme}
This section focuses on deriving compact finite difference schemes for the backward fractional Feynman-Kac equation, which are of
$q$-th $(q=1,2,3,4)$ order approximation in temporal direction and fourth order approximation in spatial direction, respectively.

Without loss of generality, consider the following backward fractional Feynman-Kac equation with non-homogeneous source term in the interval $\Omega=(0,l)$,
\begin{equation} \label{eq:equation}
\begin{aligned}
{}^{s}_{c}D_t^\alpha P(x,t):={}^{s}D_t^\alpha \left[ P(x,t)-e^{-\rho U(x)t}P(x,0)\right ]=K_{\alpha}\frac{\partial^{2}}{\partial x^{2}}P(x,t)
+f(x,t),\\
 0<t\leq T,\ x\in \Omega,
 \end{aligned}
\end{equation}
and the initial and boundary conditions are given as
\begin{align}\label{eq:inital}
&P(x,0)=\varphi (x),\, x \in \Omega,\\  \label{eq:boundary}
&P(0,t)=\psi_{1}(t),\, P(l,t)=\psi_{2}(t),\, 0<t\leq T.
\end{align}

Let $M,N$ be two positive integers, and $h=l/M, \tau=T/N$ be the uniform size of spatial grid and time step,
respectively. Then a spatial and temporal partition can be defined as $x_i=ih$ for $i=0,1,\ldots,M$, and $t_n=n\tau$ for $n=0,1,\ldots,N$.
%Take $P_i^n$ as the approximated value of $P(x_i,t_n)$, and $f_i^n=f(x_i,t_n)$.
Denote $\Omega_h=\{x_i \mid 0 \leq i \leq M \}$ and
$\Omega_{\tau}=\{t_n \mid 0 \leq n \leq N \}$.
Take $\mathcal{V}_h=\{u \mid u=(u_0,u_1,\ldots,u_M),u_0=u_M=0 \}$ as
grid function space on $\Omega_h$.
Then for any grid function $u \in \mathcal{V}_h$, we list the following notations
$$\delta_x u_{i-\frac{1}{2}}=\frac{1}{h}(u_i-u_{i-1}),\quad \delta_x^2 u_i=\frac{1}{h}(\delta_x u_{i+\frac{1}{2}}-\delta_x u_{i-\frac{1}{2}}),$$
\begin{equation*}
\mathcal{H}_h u_i=
\left\{
\begin{aligned}
& \frac{1}{12}(u_{i+1}+10u_i+u_{i-1}),&1 \leq i \leq M-1,\\
            & u_i, & i=0 \quad or \quad M.
\end{aligned}\right.
\end{equation*}
It is obvious that $\mathcal{H}_h u_i=(1+\frac{h^2}{12} \delta_x^2)u_i$ for $1 \leq i \leq M-1$.
For any $u, v \in \mathcal V_h$,
we introduce
the inner product and norms as follows
%\cite{Gao2011},
$$\langle u,v \rangle=h \sum_{i=0}^{M-1}(\delta_x u_{i+\frac{1}{2}})(\delta_x \overline{v}_{i+\frac{1}{2}})-
\frac{h^2}{12}h\sum_{i=1}^{M-1}(\delta_x^2 u_i)(\delta_x^2 \overline{v}_i),$$
$$\|u\|_{\infty}=\max_{1 \leq i \leq M-1}|u_i|,\qquad
\|\delta_x u\|=\sqrt{h\sum_{i=1}^{M}|\delta_x u_{i-\frac{1}{2}}|^2},$$
$$\|\delta_x^2 u\|=\sqrt{h\sum_{i=1}^{M-1}|\delta_x^2 u_{i}|^2},\qquad
\|u\|=\sqrt{h\sum_{i=1}^{M-1}|u_i|^2}.$$
%$\quad \|u\|=\sqrt{h\sum_{i=1}^{M-1}u_i^2}$

\begin{lemma}\label{lem:verin}
%{\cite{Gao2011}}
For $\forall u \in \mathcal{V}_h$, we have
\begin{equation}\label{eq:esinn}
\frac{2}{3}\|\delta_x u\|^2 \leq \langle u,u \rangle \leq \|\delta_x u\|^2.
\end{equation}
\end{lemma}
\begin{proof}
From the definition of inner product $\langle u,v \rangle$, we have
$\langle u,u \rangle = \|\delta_x u\|^2-\frac{h^2}{12}\|\delta_x^2 u\|^2$,
which shows $\langle u,u \rangle \leq \|\delta_x u\|^2$ immediately.

Since
\begin{equation}
h^2\|\delta_x^2u\|^2=h\sum_{i=1}^{M-1}|\delta_x u_{i+\frac{1}{2}}-\delta_x u_{i-\frac{1}{2}}|^2
\leq 2h\sum_{i=1}^{M-1}\left(|\delta_xu_{i+\frac{1}{2}}|^2+|\delta_xu_{i-\frac{1}{2}}|^2\right)
\leq 4\|\delta_xu\|^2,\nonumber
\end{equation}
we conclude that~\eqref{eq:esinn} holds.
\end{proof}

\begin{lemma}\label{lem:verinf}
{\cite{SamarskiiAndreev1976}} For $\forall u \in \mathcal{V}_h, \|u\|_{\infty} \leq \frac{\sqrt{l}}{2}\|\delta_x u\|.$
\end{lemma}

\begin{lemma}\label{lem:spdis}
{\cite{LiaoSun2010}} Let function $g(x) \in C^6\lbrack a,b \rbrack $ and $\xi (\lambda)=5(1-\lambda)^3-3(1-\lambda)^5.$ Then
$$\mathcal{H}_h g''(x_i)=\delta_x^2 g(x_i)+\frac{h^4}{360}\int_0^1[g^{(6)}(x_i-\lambda h)+g^{(6)}(x_i+\lambda h)]\xi(\lambda) d\lambda,\quad 1\leq i \leq M-1.$$
\end{lemma}

According to \cite{ChenDeng2013}, fractional substantial derivatives appeared in~\eqref{eq:equation} have $q$-th order approximations, i.e.,
\begin{equation}\label{eq:tdis}
\begin{aligned}
&{}^{s}D_t^\alpha P(x,t)|_{(x_i,t_n)}=\frac{1}{\tau^{\alpha}}\sum_{k=0}^n d_{i,k}^{q,\alpha}P(x_i,t_{n-k})+\mathcal{O}(\tau^q),\\
&{}^{s}D_t^\alpha [e^{-\rho U(x)t}P(x,0)]_{(x_i,t_n)}=\frac{1}{\tau^\alpha}\sum_{k=0}^{n}d_{i,k}^{q,\alpha}e^{-\rho U_i (n-k)\tau}P(x_i,0)+\mathcal{O}(\tau^q),\\
\end{aligned}
\end{equation}
where
$$U_i=U(x_i)$$
and
\begin{equation}\label{eq:coeff}
d_{i,k}^{q,\alpha}=e^{-\rho U_ik\tau}l_k^{q,\alpha}, q=1,2,3,4.
\end{equation}
$l_k^{1,\alpha},l_k^{2,\alpha},l_k^{3,\alpha}$ and $l_k^{4,\alpha}$ are defined by (2.2), (2.4), (2.6) and (2.8) in \cite{ChenDeng2013a},
respectively. We denote $\left(R_t^\alpha\right)_i^n=\mathcal{O}(\tau^q)$.

Consider equation~\eqref{eq:equation} at the point $(x_i,t_n)$, and we write it as following
$${}_c^{s}D_t^\alpha P(x_i,t_n)=K_\alpha \frac{\partial^2 P(x_i,t_n)}{\partial x^2}+f(x_i,t_n).$$
Acting the compact operator $\mathcal{H}_h$ on both sides of the equation above, we have
\begin{equation}
\mathcal{H}_h \big( {}_c^{s}D_t^\alpha P(x_i,t_n)\big)=K_\alpha\mathcal{H}_h\frac{\partial^2 P(x_i,t_n)}{\partial x^2}
+\mathcal{H}_h f(x_i,t_n).
\end{equation}
%Define the grid functions
%$$U_i^n=u(x_i,t_n),\quad g_i^n=g(x_i,t_n),\quad 0\leq i \leq M, \quad 0 \leq k \leq N.$$
Assuming $u(x,t)\in C_{6,2}^{x,t}([0,l]\times [0,T])$, by use of Lemma~\ref{lem:spdis} and~\eqref{eq:tdis} we obtain
\begin{equation}\label{eq:dis}
\begin{aligned}
&\mathcal{H}_h\bigg(\frac{1}{\tau^\alpha}\sum_{k=0}^n d_{i,k}^{q,\alpha}P(x_i,t_{n-k})-\frac{1}{\tau^\alpha}\sum_{k=0}^{n}d_{i,k}^{q,\alpha}e^{-\rho U_i(n-k)\tau}P(x_i,0)\bigg)\\
=&K_\alpha \delta_x^2 P(x_i,t_n)+\mathcal{H}_hf(x_i,t_n)+R_i^n,
\end{aligned}
\end{equation}
where $$R_i^n=\mathcal{H}_h(R_t^\alpha)_i^n+K_\alpha(R_x)_i^n$$
with
$$\left(R_x\right)_i^n=\frac{h^4}{360}\int_0^1\left[\frac{\partial^6P}{\partial x^6}(x_i-\lambda h,t_n)
+\frac{\partial^6P}{\partial x^6}(x_i+\lambda h,t_n)\right]\xi(\lambda)d\lambda.$$
Then there exists a constant $\widetilde C$ such that
\begin{equation}\label{eq:terror}
|R_i^n|\leq \widetilde C(\tau^q+h^4),\quad q=1,2,3,4.
\end{equation}

Denote by $P_i^n$ the approximated value of $P(x_i,t_n)$, and $f_i^n=f(x_i,t_n)$.
Multiplying~\eqref{eq:dis} by $\tau^\alpha$, and omitting
the small term, we derive the compact finite difference
schemes for solving the
backward Feynman-Kac equation~\eqref{eq:equation} with the initial condition~\eqref{eq:inital}
and boundary condition~\eqref{eq:boundary} as follows:
\begin{align}\label{eq:fulldis}
&\mathcal{H}_h\sum_{k=0}^{n}d_{i,k}^{q,\alpha}P_i^{n-k}-\mathcal{H}_h\sum_{k=0}^nd_{i,k}^{q,\alpha}e^{-\rho U_i(n-k)\tau}P_i^0 \nonumber\\
=&K_\alpha \tau^\alpha\delta_x^2P_i^n+\tau^\alpha\mathcal{H}_hf_i^n,\quad 1 \leq i \leq M-1,\quad 1 \leq n \leq N,\\
&P_i^0=\varphi(x_i),\quad 1 \leq i \leq M-1, \\  \label{eq:disbound}
&P_0^n=\psi_1(t_n),\quad P_M^n=\psi_2(t_n),\quad 1 \leq n \leq N,\\
&q=1,2,3,4.\nonumber
\end{align}

To execute the procedure, we rewrite equation~\eqref{eq:fulldis} as the following equivalent form
\begin{equation}\label{eq:execu}
\begin{aligned}
&\mathcal{H}_hd_{i,0}^{q,\alpha}P_i^n-\frac{K_\alpha\tau^\alpha}{h^2}\left(P_{i+1}^n-2P_i^n+P_{i-1}^n\right)\\
=&\sum_{k=0}^{n-1}\mathcal{H}_hd_{i,k}^{q,\alpha}e^{-\rho U_i(n-k)\tau}P_i^0
-\sum_{k=1}^{n-1}\mathcal{H}_hd_{i,k}^{q,\alpha}P_i^{n-k}+\tau^\alpha\mathcal{H}_hf_i^0,\\
&q=1,2,3,4
\end{aligned}
\end{equation}
with $i=1,2,\ldots,M-1$. It is necessary to point out that when $n=1$, the second term on the
right hand side of~\eqref{eq:execu} vanishes automatically.

%(and proving the stability and convergence for the first order time discretization scheme rigorously.)

\section{Stability and convergence analysis of the difference scheme}\label{sec:analysis}
In this section, we restrict $U(x)=1$ and do the detailed theoretical analysis
%in the case $q=1$ and $U(x)=1$.
for the first order discretization in temporal direction of schemes~\eqref{eq:fulldis}$-$\eqref{eq:disbound}.

In the following, we introduce some lemmas first, and then prove the scheme is unconditionally stable and convergent in
discrete $L^\infty$ norm.
%under the assumption $\mathfrak{R}(\rho)>0$.
For the simplification,
we denote $d_{i,k}^{1,\alpha}$ as $d_{i,k}$ and $l_k^{1,\alpha}$ as
$l_k$, respectively.

\begin{lemma}\label{lem:coeff}
{\cite{DengChenBarkai2015}} The coefficients $l_k$ defined by (2.2) in \cite{ChenDeng2013a} satisfy
\begin{equation}
l_0=1;\quad l_k<0,(k\geq 1); \quad \sum_{k=0}^{n-1}l_k>0;\quad \sum_{k=0}^\infty l_k=0;
\end{equation}
and
\begin{equation}\label{eq:estic}
\frac{1}{n^\alpha \Gamma(1-\alpha)}<\sum_{k=0}^{n-1} l_k=-\sum_{k=n}^\infty l_k \leq \frac{1}{n^\alpha}, \quad n \geq 1.
\end{equation}
\end{lemma}

\begin{theorem}\label{thm:stability}
%When $q=1$,
The difference scheme~\eqref{eq:fulldis}$-$\eqref{eq:disbound}
is unconditionally stable
%to the initial values
with the assumption
$\mathfrak{R}(\rho)>0$.
%$0 \leq \rho U_i \leq \eta$ .
\end{theorem}
\begin{proof}
Assume $\widetilde P_i^n$ is the approximate solution of $P_i^n$, which is the exact solution of the scheme~\eqref{eq:fulldis}$-$\eqref{eq:disbound}.
Let
$\varepsilon_i^n=\widetilde P_i^n-P_i^n,\ 0 \leq i\leq M,\ \ 0 \leq n\leq N$. From~\eqref{eq:fulldis}$-$\eqref{eq:disbound},
we have the perturbation error equations
\begin{align}\label{eq:stab}
&\mathcal H_h\left(d_{i,0}\varepsilon_i^n+\sum_{k=1}^{n-1}d_{i,k}\varepsilon_i^{n-k}-\sum_{k=0}^{n-1}d_{i,k}e^{-\rho (n-k)\tau}\varepsilon_i^0\right)
=K_\alpha \tau^\alpha \delta_x^2\varepsilon_i^n,\nonumber\\
&1\leq i\leq M-1,\qquad 1\leq n\leq N,\\  \label{eq:boundpet}
&\varepsilon_0^n=\varepsilon_M^n=0,\qquad 1\leq n\leq N.
%&\varepsilon_i^0=0,\qquad 0\leq i\leq M.
\end{align}
By use of~\eqref{eq:coeff}, equation~\eqref{eq:stab} can also be written as
\begin{equation}\label{eq:stabl}
\mathcal H_h\left(l_0\varepsilon_i^n+\sum_{k=1}^{n-1}e^{-\rho k\tau}l_k\varepsilon_i^{n-k}-\sum_{k=0}^{n-1}e^{-\rho n\tau}l_k\varepsilon_i^0\right)
=K_\alpha \tau^\alpha \delta_x^2\varepsilon_i^n.
\end{equation}
Multiplying~\eqref{eq:stabl} by $h(-\delta_x^2\overline{\varepsilon}_i^n)$ and summing up for $i$ from 1 to $M-1$, we get
\begin{equation}
\begin{aligned}
&h\sum_{i=1}^{M-1}\left(-\delta_x^2\overline{\varepsilon}_i^n\right)\left[\left(1+\frac{h^2}{12}\delta_x^2\right)\left(l_0\varepsilon_i^n+
\sum_{k=1}^{n-1}e^{-\rho k\tau}l_k\varepsilon_i^{n-k}-\sum_{k=0}^{n-1}e^{-\rho n\tau}l_k\varepsilon_i^0\right)\right]\\
=&h\sum_{i=1}^{M-1}\left(-K_\alpha\tau^\alpha\right)|\delta_x^2\varepsilon_i^n|^2.
\end{aligned}
\end{equation}
Using the summation formula by parts and noticing~\eqref{eq:boundpet}, we obtain
\begin{equation}\nonumber
\begin{aligned}
&l_0h\sum_{i=0}^{M-1}\left|\delta_x\varepsilon_{i+\frac{1}{2}}^n\right|^2-l_0\frac{h^2}{12}h\sum_{i=1}^{M-1}\left|\delta_x^2\varepsilon_i^n\right|^2\\
=&-h\sum_{i=0}^{M-1}\sum_{k=1}^{n-1}l_k\left(\delta_x\overline{\varepsilon}_{i+\frac{1}{2}}^n\right)\left(\delta_x\left(e^{-\rho k\tau}\varepsilon_{i+\frac{1}{2}}^{n-k}\right)\right)\\
&+\frac{h^2}{12}h\sum_{i=1}^{M-1}\sum_{k=1}^{n-1}l_k
\left(\delta_x^2\overline{\varepsilon}_i^n\right)\left(\delta_x^2\left(e^{-\rho k\tau}\varepsilon_i^{n-k}\right)\right)\\
&+h\sum_{i=0}^{M-1}\sum_{k=0}^{n-1}l_k\left(\delta_x\overline{\varepsilon}_{i+\frac{1}{2}}^n\right)\left(\delta_x\left(e^{-\rho n\tau}\varepsilon_{i+\frac{1}{2}}^0\right)\right)\\
&-\frac{h^2}{12}h\sum_{i=1}^{M-1}\sum_{k=0}^{n-1}l_k\left(\delta_x^2\overline{\varepsilon}_i^n\right)
\left(\delta_x^2\left(e^{-\rho n\tau}\varepsilon_i^0\right)\right)
-K_\alpha\tau^\alpha\parallel\delta_x^2\varepsilon^n\parallel^2.
\end{aligned}
\end{equation}
Then it can be deduced immediately that the inequality below holds.
\begin{equation}\nonumber
l_0\langle\varepsilon^n,\varepsilon^n\rangle\leq -\sum_{k=1}^{n-1}l_k\left|\langle e^{-\rho k\tau}\varepsilon^{n-k},\varepsilon^n\rangle\right|
+\sum_{k=0}^{n-1}l_k\left|\langle e^{-\rho n\tau}\varepsilon^0 ,\varepsilon^n\rangle\right|.
\end{equation}

Let
$$A_1=-\sum_{k=1}^{n-1}l_k\left|\langle e^{-\rho k\tau}\varepsilon^{n-k},\varepsilon^n\rangle\right|$$
and $$A_2=\sum_{k=0}^{n-1}l_k\left|\langle e^{-\rho n\tau}\varepsilon^0 ,\varepsilon^n\rangle\right|.$$
%Whereafter we estimate $A_1$ and $A_2$, respectively.
From the Cauchy-Schwarz inequality and Lemma~\ref{lem:coeff}, we have the estimates
\begin{equation}\nonumber
A_1\leq \frac{1}{2}\sum_{k=1}^{n-1}(-l_k)\left(\langle \varepsilon^n,\varepsilon^n\rangle
+\langle e^{-\rho k\tau}\varepsilon^{n-k},e^{-\rho k\tau}\varepsilon^{n-k}\rangle\right),
\end{equation}
\begin{equation}\nonumber
A_2\leq \frac{1}{2}\sum_{k=0}^{n-1}l_k\left(\langle\varepsilon^n,\varepsilon^n\rangle+\langle e^{-\rho n\tau}\varepsilon^0,e^{-\rho n\tau}\varepsilon^0\rangle\right).
\end{equation}
It follows that
\begin{equation}\label{eq:verpp}
\frac{1}{2}l_0\langle\varepsilon^n,\varepsilon^n\rangle\leq
-\frac{1}{2}\sum_{k=1}^{n-1}l_k\langle\varepsilon^{n-k},\varepsilon^{n-k}\rangle+\frac{1}{2}l_0\langle\varepsilon^0,\varepsilon^0\rangle
+\frac{1}{2}\sum_{k=1}^{n-1}l_k\langle\varepsilon^0,\varepsilon^0\rangle.
\end{equation}

Next, we prove
\begin{equation}\label{eq:verp}
\langle\varepsilon^n,\varepsilon^n\rangle\leq\langle\varepsilon^0,\varepsilon^0\rangle
\end{equation}
by mathematical induction.
In the case $n=1$,~\eqref{eq:verp} holds obviously according to~\eqref{eq:verpp}.
Suppose that for $s=1,2,\dots,n-1$,
\begin{equation}\label{eq:supp}
\langle\varepsilon^s,\varepsilon^s\rangle\leq\langle\varepsilon^0,\varepsilon^0\rangle
\end{equation}
holds.
When $s=n$, according to~\eqref{eq:verpp} and~\eqref{eq:supp}, we have
\begin{equation}\label{eq:var}\nonumber
\frac{1}{2}l_0\langle\varepsilon^n,\varepsilon^n\rangle\leq
-\frac{1}{2}\sum_{k=1}^{n-1}l_k\langle\varepsilon^0,\varepsilon^0\rangle
+\frac{1}{2}l_0\langle\varepsilon^0,\varepsilon^0\rangle
+\frac{1}{2}\sum_{k=1}^{n-1}l_k\langle\varepsilon^0,\varepsilon^0\rangle,
\end{equation}
which indicates that $\langle\varepsilon^n,\varepsilon^n\rangle\leq\langle\varepsilon^0,\varepsilon^0\rangle.$

Combining~\eqref{eq:verp} with Lemma~\ref{lem:verin} and Lemma~\ref{lem:verinf}, we conclude that
\begin{equation}\nonumber
\begin{aligned}
\parallel\varepsilon^n\parallel_\infty^2\leq & \frac{l}{4}\|\delta_x \varepsilon^n\|^2
\leq  \frac{l}{4}\cdot\frac{3}{2}\langle\varepsilon^n,\varepsilon^n\rangle
\leq  \frac{3l}{8}\langle\varepsilon^0,\varepsilon^0\rangle
\leq  \frac{3l}{8}\|\delta_x\varepsilon^0\|^2,
\end{aligned}
\end{equation}
which completes the proof.
%implies that the difference scheme~\eqref{eq:fulldis}-~\eqref{eq:disbound} is unconditionally stable to the initial values in maximum norm.
\end{proof}

\begin{theorem}\label{thm:abc}
Let $P_i^n$ be the solution of the difference scheme~\eqref{eq:fulldis}$-$\eqref{eq:disbound}, and $P(x_i,t_n)$ be the solution of the
problem~\eqref{eq:equation}$-$\eqref{eq:boundary} with the assumption $\mathfrak{R}(\rho)>0$.
Denote $E_i^n=P(x_i,t_n)-P_i^n,\ 0\leq i\leq M,\ 0\leq n\leq N$,
then there exists a positive constant $C$ such that
\begin{equation}
\parallel E^n \parallel_\infty\leq C\left(\tau+h^4\right),\qquad 0\leq n\leq N.
\end{equation}
%if the assumption $0\leq\rho U_i\leq \eta$ holds.
\end{theorem}
\begin{proof}
%Let $E_i^n=P(x_i,t_n)-P_i^n,\ 0\leq i\leq M,\ 0\leq n\leq N$.
According to~\eqref{eq:dis} and~\eqref{eq:fulldis}$-$\eqref{eq:disbound}, we get the error equations
\begin{align}
&\mathcal H_h\left(d_{i,0}E_i^n+\sum_{k=1}^{n-1}d_{i,k}E_i^{n-k}\right)
=K_\alpha\tau^\alpha\delta_x^2E_i^n+\tau^\alpha R_i^n, \nonumber\\
&1\leq i\leq M-1, \qquad 1\leq k\leq N,\label{eq:error}\\
&E_0^n=0,\quad E_M^n=0, \quad 1\leq n \leq N,\\
&E_i^0=0,\quad 1\leq i\leq M-1
\end{align}
with $\mid R_i^n\mid\leq \widetilde C(\tau+h^4)$ given in~\eqref{eq:terror}.
Multiplying~\eqref{eq:error} by $h(-\delta_x^2\overline
{E}_i^n)$ and summing up for $i$ from 1 to $M-1$, we have
\begin{equation}\label{eq:multi}
\begin{aligned}
&h\sum_{i=1}^{M-1}(-\delta_x^2\overline{E}_i^n)\left[\left(1+\frac{h^2}{12}\delta_x^2\right)\left(d_{i,0}E_i^n+\sum_{k=1}^{n-1}d_{i,k}E_i^{n-k}\right)\right]\\
=&hK_\alpha\tau^\alpha\sum_{i=1}^{M-1}\left(-\delta_x^2\overline{E}_i^n\right)\left(\delta_x^2E_i^n\right)
+h\tau^\alpha\sum_{i=1}^{M-1}\left(-\delta_x^2\overline{E}_i^n\right)R_i^n.
\end{aligned}
\end{equation}
From~\eqref{eq:coeff},~\eqref{eq:multi} can also be written as
\begin{equation}\label{eq:formu}
\begin{aligned}
&h\sum_{i=1}^{M-1}\left(-\delta_x^2\overline{E}_i^n\right)\left(l_0E_i^n+\sum_{k=1}^{n-1}e^{-\rho k\tau}l_kE_i^{n-k}\right)\\
&+\frac{h^2}{12}h\sum_{i=1}^{M-1}\left(-\delta_x^2\overline{E}_i^n\right)\left(l_0\delta_x^2E_i^n+
\sum_{k=1}^{n-1}l_k\delta_x^2\left(e^{-\rho k\tau}E_i^{n-k}\right)\right)\\
=&-K_\alpha\tau^\alpha h\sum_{i=1}^{M-1}\left|\delta_x^2E_i^n\right|^2-\tau^\alpha h\sum_{i=1}^{M-1}\left(\delta_x^2\overline{E}_i^n\right)R_i^n.
\end{aligned}
\end{equation}
Using the summation formula by parts, we obtain from~\eqref{eq:formu}
\begin{equation}\nonumber
\begin{aligned}
&l_0h\sum_{i=0}^{M-1}\left|\delta_xE_{i+\frac{1}{2}}^n\right|^2+h\sum_{k=1}^{n-1}l_k\sum_{i=0}^{M-1}\left(\delta_x\overline{E}_{i+\frac{1}{2}}^n\right)
\left(\delta_xe^{-\rho k\tau}E_{i+\frac{1}{2}}^{n-k}\right)\\
&-l_0\frac{h^2}{12}h\sum_{i=1}^{M-1}\left|\delta_x^2E_i^n\right|^2-\frac{h^2}{12}h\sum_{k=1}^{n-1}l_k\sum_{i=1}^{M-1}
\left(\delta_x^2\overline{E}_i^n\right)\left(\delta_x^2e^{-\rho k\tau}E_i^{n-k}\right)\\
=&-K_\alpha\tau^\alpha h\sum_{i=1}^{M-1}\left|\delta_x^2E_i^n\right|^2-\tau^\alpha h\sum_{i=1}^{M-1}\left(\delta_x^2\overline{E}_i^n\right)R_i^n,
\end{aligned}
\end{equation}
which implies
\begin{equation}\nonumber
l_0\langle E^n,E^n\rangle+\sum_{k=1}^{n-1}l_k\langle e^{-\rho k\tau}E^{n-k}, E^n\rangle
=-K_\alpha\tau^\alpha\parallel\delta_x^2E^n\parallel^2-\tau^\alpha h\sum_{i=1}^{M-1}\left(\delta_x^2\overline{E}_i^n\right)R_i^n.
\end{equation}
Then it can be deduced that
\begin{equation}\label{eq:esti}
\begin{aligned}
l_0\langle E^n,E^n\rangle=&\left|-\sum_{k=1}^{n-1}l_k\langle e^{-\rho k\tau}E^{n-k},E^n\rangle
-\tau^\alpha h\sum_{i=1}^{M-1}\left(\delta_x^2\overline{E}_i^n\right)R_i^n\right|
-K_\alpha\tau^\alpha\parallel\delta_x^2E^n\parallel^2\\
\leq &\left|-\sum_{k=1}^{n-1}l_k\langle e^{-\rho k\tau}E^{n-k},E^n\rangle\right|
+\tau^\alpha h\left|\sum_{i=1}^{M-1}\left(\delta_x^2\overline{E}_i^n\right)R_i^n\right|
-K_\alpha\tau^\alpha\parallel\delta_x^2E^n\parallel^2\\
\leq &\sum_{k=1}^{n-1}(-l_k)\left|\langle e^{-\rho k\tau}E^{n-k},E^n\rangle\right|
+\tau^\alpha h\left|\sum_{i=1}^{M-1}\left(\delta_x^2\overline{E}_i^n\right)R_i^n\right|
-K_\alpha\tau^\alpha\parallel\delta_x^2E^n\parallel^2\\
\leq &-\frac{1}{2}\sum_{k=1}^{n-1}l_k\left(\langle E^n,E^n\rangle+\langle e^{-\rho k\tau}E^{n-k},e^{-\rho k\tau}E^{n-k}\rangle\right)\\
&-K_\alpha\tau^\alpha\parallel\delta_x^2E^n\parallel^2
+\tau^\alpha h\left|\sum_{i=1}^{M-1}\left(\delta_x^2\overline{E}_i^n\right)R_i^n\right|\\
%\leq &-\frac{1}{2}\sum_{k=1}^{n-1}l_k\langle E^n,E^n\rangle-\frac{1}{2}\sum_{k=1}^{n-1}l_k\langle E^{n-k},E^{n-k}\rangle\\
%&-K_\alpha\tau^\alpha\parallel\delta_x^2E^n\parallel^2
%-\tau^\alpha h\sum_{i=1}^{M-1}\left(\delta_x^2E_i^n\right)R_i^n.
\end{aligned}
\end{equation}
with the Cauchy-Schwarz inequality being used.

Let $$B=\tau^\alpha h\left|\sum_{i=1}^{M-1}\left(\delta_x^2\overline{E}_i^n\right)R_i^n\right|.$$
%Next, we do the estimate of $B$.
It is clear that
\begin{equation}\label{eq:estib}
\begin{aligned}
B\leq &\tau^\alpha h\sum_{i=1}^{M-1}\left|\left(\delta_x^2\overline{E}_i^n\right)R_i^n\right|\\
\leq &\tau^\alpha hK_\alpha\sum_{i=1}^{M-1}\left|\delta_x^2\overline{E}_i^n\right|^2+\tau^\alpha h\frac{1}{4K_\alpha}\sum_{i=1}^{M-1}\left|R_i^n\right|^2\\
=&K_\alpha\tau^\alpha\|\delta_x^2E^n\|^2+\frac{\tau^\alpha}{4K_\alpha}\| R^n\|^2.
\end{aligned}
\end{equation}
Since $\mid R_i^n\mid\leq \widetilde C\left(\tau+h^4\right)$, it can be obtained that
\begin{equation}\label{eq:estir}
\| R^n\|^2\leq l\widetilde C^2(\tau+h^4)^2.
\end{equation}
Substituting~\eqref{eq:estib} into~\eqref{eq:esti}, and noticing~\eqref{eq:estir}, we have
\begin{equation}\label{eq:forr}
l_0\langle E^n,E^n\rangle\leq-\frac{1}{2}\sum_{k=1}^{n-1}l_k\langle E^n,E^n\rangle-\frac{1}{2}\sum_{k=1}^{n-1}l_k\langle E^{n-k},E^{n-k}\rangle
+\frac{C_1\tau^\alpha}{2}(\tau+h^4)^2,
\end{equation}
where
$C_1=\frac{\widetilde C^2l}{2K_\alpha}$.
%$\frac{C_1}{2}=\frac{\widetilde C^2(b-a)}{4K_\alpha}$.
From Lemma~\ref{lem:coeff}, there exists $$0<-\frac{1}{2}\sum_{k=1}^{n-1}l_k<\frac{1}{2},$$
then it is derived immediately from~\eqref{eq:forr} that
\begin{equation}\label{eq:ineq}
\langle E^n,E^n\rangle\leq-\sum_{k=1}^{n-1}l_k\langle E^{n-k},E^{n-k}\rangle+C_1\tau^\alpha(\tau+h^4)^2.
\end{equation}

In the following, we prove
\begin{equation}\label{eq:inequ}
\langle E^n,E^n\rangle\leq\left(\sum_{k=0}^{n-1}l_k\right)^{-1}C_1\tau^\alpha(\tau+h^4)^2
\end{equation}
by mathematical induction.

For $n=1$,~\eqref{eq:inequ} holds by~\eqref{eq:ineq}. Suppose $$\langle E^s,E^s\rangle\leq\left(\sum_{k=0}^{s-1}l_k\right)^{-1}C_1\tau^\alpha(\tau+h^4)^2$$
when $s=1,2,\ldots,n-1$.
Then for $s=n$, by~\eqref{eq:ineq} and the assumption we conclude that
\begin{equation}\label{eq:estif}
\begin{aligned}
\langle E^n,E^n\rangle &\leq-\sum_{k=1}^{n-1}l_k\left(\sum_{m=0}^{n-k-1}l_m\right)^{-1}C_1\tau^\alpha\left(\tau+h^4\right)^2
+C_1\tau^\alpha\left(\tau+h^4\right)^2\\
&\leq-\left(\sum_{k=0}^{n-1}l_k\right)^{-1}\sum_{k=1}^{n-1}l_kC_1\tau^\alpha\left(\tau+h^4\right)^2+C_1\tau^\alpha(\tau+h^4)^2\\
&=\left(1-\frac{\sum_{k=1}^{n-1}l_k}{\sum_{k=0}^{n-1}l_k}\right)C_1\tau^\alpha(\tau+h^4)^2\\
&=\left(\sum_{k=0}^{n-1}l_k\right)^{-1}C_1\tau^\alpha(\tau+h^4)^2.
\end{aligned}
\end{equation}

Finally, according to~\eqref{eq:estif},~\eqref{eq:estic}, Lemma~\ref{lem:verin} and Lemma~\ref{lem:verinf}, we derive
\begin{equation}
\begin{aligned}
\| E^n\|_\infty^2&\leq\frac{l}{4}\|\delta_x E^n\|^2\\
&\leq\frac{3l}{8}\langle E^n,E^n\rangle\\
&\leq\frac{3l}{8}\left(\sum_{k=0}^{n-1}l_k\right)^{-1}C_1\tau^\alpha(\tau+h^4)^2\\
&\leq\frac{3l}{8}n^\alpha\Gamma(1-\alpha)C_1\tau^\alpha(\tau+h^4)^2\\
&\leq\lbrack C\left(\tau+h^4\right)\rbrack^2,
\end{aligned}
\end{equation}
where $C=\sqrt{\frac{3l}{8}T^\alpha\Gamma(1-\alpha)C_1}$.
\end{proof}
%\begin{remark}
%If $\rho$ is an imaginary number, i.e., $\rho=a+ib$, $a,b \in \mathbb{R}$, the assumption $0\leq \rho U_i\leq \eta$ should be instead by
%$0\leq aU_i\leq \eta$. By the similar proof of Theorem~\ref{thm:stability} and Theorem~\ref{thm:abc}, the results of stability and convergence can be
%obtained accordingly.
%\end{remark}

\section{Numerical examples}\label{sec:examples}
In this section, we consider some numerical examples to demonstrate the
effectiveness of the schemes, and verify the theoretical results including convergence orders and
numerical stability.
%To illustrate the numerical results, we use
The discrete $L^\infty$ norm is used to measure the numerical errors,
which makes the results stronger than the discrete $L^2$ norm does.

In the following examples, denote $i=\sqrt{-1}$, and we choose $U(x)=1$ and $x$, respectively.
\begin{example}\label{exam:a}
%In~\eqref{eq:equation}$-$\eqref{eq:boundary}, take $l=1$, $T=1$,
%$\varphi(x)=\sin(\pi x)$, $\psi_1(t)=\psi_2(t)=0$, and
%$$f(x,t)=\frac{\Gamma(4+\alpha)}{\Gamma(4)}e^{-\rho t}t^3\sin(\pi x)
%+K_\alpha \pi^2\left(t^{3+\alpha}+1\right)e^{-\rho t}\sin(\pi x).$$
For the backward fractional Feynman-Kac equation~\eqref{eq:equation} on the finite domain $0<x<1$,
$0<t\leq 1$,
take
%$\Omega=(0,1)$, $T=1$,
$K_\alpha=0.5$, $U(x)=1$,
$\rho=1+i$,
the forcing function
\begin{equation}
\begin{aligned}
f(x,t)=&\frac{\Gamma(4+\alpha)}{\Gamma(4)}e^{-\rho t}t^3\sin(\pi x)
+K_\alpha \pi^2\left(t^{3+\alpha}+1\right)e^{-\rho t}\sin(\pi x),\nonumber
\end{aligned}
\end{equation}
with the initial condition $P(x,0)=\sin(\pi x)$, and the boundary conditions $P(0,t)=P(1,t)=0$.
The exact solution is given by
%The exact solution is given by
\begin{equation}
P(x,t)=e^{-\rho t}\left(t^{3+\alpha}+1\right)\sin(\pi x).\nonumber
\end{equation}
\end{example}

\begin{table}[h!]
\caption{The maximum errors and convergence orders for Example~\ref{exam:a} in temporal direction with $q=1$ and $h=1/1000$.}\label{table:a}
  \centering
  \begin{tabular}{ccccccc}
    \toprule
    $\tau$ & $\alpha=0.2 $  & Rate   & $\alpha=0.5 $ & Rate   & $ \alpha=0.8$ & Rate \\
    \midrule
      1/10   & 0.0023         &      - & 0.0080        &      - & 0.0182        &      -\\
      1/20   & 0.0011         & 1.0641 & 0.0041        & 0.9644 & 0.0093        & 0.9686\\
      1/40   & 5.7830e-004    & 0.9276 & 0.0020        & 1.0356 & 0.0047        & 0.9846\\
      1/80   & 2.9000e-004    & 0.9958 & 0.0010        & 1.0000 & 0.0024        & 0.9696\\
    \bottomrule
  \end{tabular}
\end{table}

\begin{table}[h!]
\caption{The maximum errors and convergence orders for Example~\ref{exam:a} in spatial direction with $q=1$ and $\tau=h^4$.}
  \centering
  \begin{tabular}{ccccccc}
    \toprule
    $h$ & $\alpha=0.2 $  & Rate   & $\alpha=0.5 $ & Rate   & $ \alpha=0.8$ & Rate \\
    \midrule
      1/2   & 0.0184         &      - & 0.0217        &      - & 0.0279        &      -\\
      1/4   & 0.0011         & 4.0641 & 0.0013        & 4.0611 & 0.0017        & 4.0367\\
      1/8   & 6.5910e-005    & 4.0609 & 7.8904e-005   & 4.0423 & 1.0316e-004   & 4.0426\\
      1/10  & 2.6942e-005    & 4.0091 & 3.2266e-005   & 4.0074 & 4.2201e-005   & 4.0057\\
    \bottomrule
  \end{tabular}
\end{table}

\begin{table}[h!]
\caption{The maximum errors and convergence orders for Example~\ref{exam:a} in spatial direction with $q=2$ and $\tau=h^2$.}
  \centering
  \begin{tabular}{ccccccc}
    \toprule
    $h$ & $\alpha=0.2 $  & Rate   & $\alpha=0.5 $ & Rate   & $ \alpha=0.8$ & Rate \\
    \midrule
      1/10   & 2.7752e-005    &      - & 3.5587e-005        &      - & 5.1361e-005        &      -\\
      1/20   & 1.7312e-006    & 4.0027 & 2.2245e-006        & 3.9998 & 3.2177e-006        & 3.9966\\
      1/40   & 1.0815e-007    & 4.0007 & 1.3903e-007        & 4.0000 & 2.0122e-007        & 3.9992\\
      1/80   & 6.7586e-009    & 4.0002 & 8.6896e-009        & 4.0000 & 1.2578e-008        & 3.9998\\
    \bottomrule
  \end{tabular}
\end{table}

\begin{table}[h!]
\caption{The maximum errors and convergence orders for Example~\ref{exam:a} in temporal direction with $q=3$ and $h=1/1000$.}\label{table:d}
  \centering
  \begin{tabular}{ccccccc}
    \toprule
    $\tau$ & $\alpha=0.2 $  & Rate   & $\alpha=0.5 $ & Rate   & $ \alpha=0.8$ & Rate \\
    \midrule
      1/10   & 2.4487e-005    &      -    & 9.5460e-005   &      -     & 2.4802e-004   &      -\\
      1/20   & 3.0669e-006    & 2.9972    & 1.1984e-005   & 2.9938     & 3.1167e-005   & 2.9924\\
      1/40   & 3.8372e-007    & 2.9987    & 1.5009e-006   & 2.9972     & 3.9044e-006   & 2.9968\\
      1/80   & 4.7992e-008    & 2.9992    & 1.8778e-007   & 2.9987     & 4.8854e-007   & 2.9986\\
      1/160  & 6.0077e-009    & 2.9979    & 2.3470e-008   & 3.0002     & 6.1094e-008   & 2.9994\\
    \bottomrule
  \end{tabular}
\end{table}

Table~\ref{table:a}$-$Table~\ref{table:d} illustrate part of the numerical results to verify the convergence orders
of schemes~\eqref{eq:fulldis}$-$\eqref{eq:disbound} for Example~\ref{exam:a}.
%verify the global truncation errors of the algorithms with $q=1$ and $q=2$.
Other cases also coincide with the results derived in Section~\ref{sec:scheme}, though we omit them here.

\begin{example}\label{exam:b}
For the backward fractional Feynman-Kac equation~\eqref{eq:equation} on the finite domain $0<x<1$,
$0<t\leq 1$,
take $K_\alpha=0.5$, $U(x)=x$,
$\rho=1+i$,
the forcing function
\begin{equation}
\begin{aligned}
f(x,t)=&\frac{\Gamma(4+\alpha)}{\Gamma(4)}e^{-\rho xt}t^3\sin(\pi x)\\
&-K_\alpha e^{-\rho xt}\left(t^{3+\alpha}+1\right)\left(\rho^2t^2\sin(\pi x)-2\pi\rho t\cos(\pi x)-\pi^2\sin(\pi x)\right),\nonumber
\end{aligned}
\end{equation}
with the initial condition $P(x,0)=\sin(\pi x)$, and the boundary conditions $P(0,t)=P(1,t)=0$.
The exact solution is given by
\begin{equation}
P(x,t)=e^{-\rho xt}\left(t^{3+\alpha}+1\right)\sin(\pi x).\nonumber
\end{equation}
\end{example}

\begin{table}[h!]
\caption{The maximum errors and convergence orders for Example~\ref{exam:b} in temporal direction with $q=1$ and $h=1/1000$.}\label{table:f}
  \centering
  \begin{tabular}{ccccccc}
    \toprule
    $\tau$ & $\alpha=0.2 $  & Rate   & $\alpha=0.5 $ & Rate   & $ \alpha=0.8$ & Rate \\
    \midrule
      1/10   & 0.0038         &      - & 0.0132        &      - & 0.0301        &      -\\
      1/20   & 0.0019         & 1.0000 & 0.0067        & 0.9783 & 0.0154        & 0.9668\\
      1/40   & 9.5861e-004    & 0.9870 & 0.0034        & 0.9786 & 0.0078        & 0.9814\\
      1/80   & 4.8072e-004    & 0.9957 & 0.0017        & 1.0000 & 0.0039        & 1.0000\\
    \bottomrule
  \end{tabular}
\end{table}

\begin{table}[h!]
\caption{The maximum errors and convergence orders for Example~\ref{exam:b} in spatial direction with $q=1$ and $\tau=h^4$.}
  \centering
  \begin{tabular}{ccccccc}
    \toprule
    $h$ & $\alpha=0.2 $  & Rate   & $\alpha=0.5 $ & Rate   & $ \alpha=0.8$ & Rate \\
    \midrule
      1/2   & 0.0771         &      -  & 0.0731        &      -  & 0.0713        &      - \\
      1/4   & 0.0043         & 4.1643  & 0.0040        & 4.1918  & 0.0040        & 4.1558 \\
      1/8   & 2.6053e-004    & 4.0448  & 2.5475e-004   & 3.9728  & 2.6205e-004   & 3.9321 \\
      1/10  & 1.0733e-004    & 3.9742  & 1.0446e-004   & 3.9951  & 1.0689e-004   & 4.0186 \\
    \bottomrule
  \end{tabular}
\end{table}

\begin{table}[h!]
\caption{The maximum errors and convergence orders for Example~\ref{exam:b} in temporal direction with $q=2$ and $h=1/1000$.}
  \centering
  \begin{tabular}{ccccccc}
    \toprule
    $\tau$ & $\alpha=0.2 $  & Rate   & $\alpha=0.5 $ & Rate   & $ \alpha=0.8$ & Rate \\
    \midrule
      1/10   & 4.8365e-004    &      -    & 0.0018        &      -     & 0.0043        &      -\\
      1/20   & 1.2599e-004    & 1.9407    & 4.6403e-004   & 1.9557     & 0.0011        & 1.9668\\
      1/40   & 3.2132e-005    & 1.9712    & 1.1850e-004   & 1.9693     & 2.8775e-004   & 1.9346\\
      1/80   & 8.1125e-006    & 1.9858    & 2.9935e-005   & 1.9850     & 7.2748e-005   & 1.9838\\
      1/160  & 2.0381e-006    & 1.9929    & 7.5227e-006   & 1.9925     & 1.8288e-005   & 1.9920\\
    \bottomrule
  \end{tabular}
\end{table}

\begin{table}[h!]
\caption{The maximum errors and convergence orders for Example~\ref{exam:b} in spatial direction with $q=2$ and $\tau=h^2$.}
  \centering
  \begin{tabular}{ccccccc}
    \toprule
    $h$ & $\alpha=0.2 $  & Rate   & $\alpha=0.5 $ & Rate   & $ \alpha=0.8$ & Rate \\
    \midrule
      1/10   & 1.0794e-004    &      -  & 1.0761e-004   &      -  & 1.1799e-004  &      - \\
      1/20   & 6.7206e-006    & 4.0055  & 6.7031e-006   & 4.0048  & 7.3687e-006  & 4.0011 \\
      1/40   & 4.2056e-007    & 3.9982  & 4.1859e-007   & 4.0012  & 4.6134e-007  & 3.9975 \\
      1/80   & 2.6277e-008    & 4.0004  & 2.6157e-008   & 4.0003  & 2.8833e-008  & 4.0000 \\
    \bottomrule
  \end{tabular}
\end{table}

\begin{table}[h!]
\caption{The maximum errors and convergence orders for Example~\ref{exam:b} in temporal direction with $q=3$ and $h=1/1000$.}
  \centering
  \begin{tabular}{ccccccc}
    \toprule
    $\tau$ & $\alpha=0.2 $  & Rate   & $\alpha=0.5 $ & Rate   & $ \alpha=0.8$ & Rate \\
    \midrule
      1/10   & 4.0581e-005    &      -   & 1.5822e-004     &      -   & 4.1108e-004       &      -\\
      1/20   & 5.0826e-006    & 2.9972   & 1.9862e-005     & 2.9938   & 5.1656e-005       & 2.9924\\
      1/40   & 6.3591e-007    & 2.9987   & 2.4876e-006     & 2.9972   & 6.4710e-006       & 2.9969\\
      1/80   & 7.9535e-008    & 2.9992   & 3.1123e-007     & 2.9987   & 8.0969e-007       & 2.9985\\
      1/160  & 9.9561e-009    & 2.9979   & 3.8898e-008     & 3.0002   & 1.0126e-007       & 2.9993\\
    \bottomrule
  \end{tabular}
\end{table}

\begin{table}[h!]
\caption{The maximum errors and convergence orders for Example~\ref{exam:b} in spatial direction with $q=3$ and $\tau=h^{4/3}$.}
  \centering
  \begin{tabular}{ccccccc}
    \toprule
    $h$ & $\alpha=0.2 $  & Rate   & $\alpha=0.5 $ & Rate   & $ \alpha=0.8$ & Rate \\
    \midrule
      1/8    & 2.6076e-004    &      -  & 2.5785e-004   &      -  & 2.9742e-004  &      - \\
      1/16   & 1.6025e-005    & 4.0243  & 1.5697e-005   & 4.0380  & 1.7149e-005  & 4.0262 \\
      1/32   & 1.0054e-006    & 3.9945  & 9.8854e-007   & 3.9890  & 1.0750e-006  & 3.9957 \\
      1/64   & 6.3848e-008    & 3.9770  & 6.2690e-008   & 3.9790  & 6.7897e-008  & 3.9848 \\
      1/128  &3.9893e-009     & 4.0004  & 3.9163e-009   & 4.0007  & 4.2434e-009  & 4.0001 \\
    \bottomrule
  \end{tabular}
\end{table}

\begin{table}[h!]
\caption{The maximum errors and convergence orders for Example~\ref{exam:b} in both temporal and spatial directions with $q=4$ and $\tau=h$.}\label{table:g}
  \centering
  \begin{tabular}{ccccccc}
    \toprule
    $h$ & $\alpha=0.2 $  & Rate   & $\alpha=0.5 $ & Rate   & $ \alpha=0.8$ & Rate \\
    \midrule
      1/10   & 1.0568e-004    &      -  & 9.7935e-005   &      -   & 8.9182e-005   &      - \\
      1/20   & 6.5976e-006    & 4.0016  & 6.0897e-006   & 4.0074   & 5.5313e-006   & 4.0111 \\
      1/40   & 4.1219e-007    & 4.0006  & 3.8077e-007   & 3.9994   & 3.4481e-007   & 4.0037 \\
      1/80   & 2.5760e-008    & 4.0001  & 2.3789e-008   & 4.0006   & 2.1531e-008   & 4.0013 \\
      1/160  & 1.6098e-009    & 4.0002  &1.4867e-009    & 4.0001   & 1.3451e-009   & 4.0006 \\
    \bottomrule
  \end{tabular}
\end{table}

Table \ref{table:f}$-$Table \ref{table:g} show the maximum errors and convergence orders
for Example~\ref{exam:b}
in both temporal and spatial directions, respectively,
which confirm the global truncation error of schemes~\eqref{eq:fulldis}$-$\eqref{eq:disbound}
is $\mathcal{O}(\tau^q+h^4)$ for $q=1,2,3,4$.
%\end{example}
%\begin{figure}[h!]
%%\centering
%\includegraphics[width=2.3in]{e.eps}%\\[5pt]
%\end{figure}

\section{Conclusion}\label{sec:conclusion}
%In this paper, We used Galerkin method to approach the nonlinear Riesz space fractional diffusion equations
%on complex convex domain by approximate nonlinear term with  Taylor formula.
%This mehtod have some advantages compare with existing methods.
%%Firstly, it can compute easily in complex domain. Secondly, it no need to solve nonlinear systems.
%It can be used to solve these problems on complex domain, which is seldom solved using current method.
%And the linearization method,
%which can be used in most situations where there is a nonlinear term,
%is a very useful approach to approximate nonlinear term.
%Also, this method can make the computation simple without loss of accuracy.
%Here, we just considered the homogenous Dirichlet boundary conditions. In the following
%work, we will consider other boundary conditions, including nonhomogenous boundary
%conditions, Neumann boundary conditions. Also, we will consider some techniques
%to improve the calculation speed.

In this paper, we construct compact finite difference schemes for solving the backward fractional Feynman-Kac
equation by its equivalent form, where the $q$-th $(q=1,2,3,4)$ order approximation operators for fractional substantial derivative
%proposed in \cite{ChenDeng2013}
are used in temporal direction, and fourth order
compact difference operator for the spatial derivative, respectively.
%The algorithm proposed here are of
%$q$-order $(q=1,2,3,4)$ and fourth-order accuracy in temporal and spatial directions, respectively, which are more
%accurate than the numerical methods presented for this kind of problem before.
%The compact difference scheme
%with first order accuracy in temporal direction and fourth order in spatial direction is theoretically analyzed
%For the case $q=1$,
By introducing a new inner product,
we prove rigorously that the scheme is unconditionally stable and convergent
%with first order accuracy in temporal direction and fourth order accuracy in spatial direction
in maximum norm when $q=1$.
For all the schemes proposed,
from first to fourth order in temporal direction,
abundant examples are performed to verify the theoretical analysis and their effectiveness.
%Applications of these techniques in solving high-dimensional fractional Feynman-Kac
%equations will be our future consideration.

%\begin{appendices}
%\section{}\label{app:A}
%% TODO Should put the following sentence in Appendix
%Here, we consider the calculation of $S_j(x_i)$.  We give it in detail by
%considering left Riemann-Liouville derivative. Suppose $f(x) = \varphi_l(x,y_i)$
%and $x_i^j \ne x$
%\end{appendices}
\section*{Acknowledgements}
This research was supported by National Natural Science Foundations of China (No.11471262).
The authors would like to express their gratitude to the referees for their very helpful
comments and suggestions on the manuscript.

% \bibliographystyle{siam}
%\section*{References}
%\bibliographystyle{plain}
%\bibliographystyle{model1b-num-names}
\bibliography{Reference}
%\printbibliography
\end{document}